\documentclass[12pt,a4]{article}
\usepackage[english]{babel}
\usepackage{newlfont}
\usepackage{a4wide}
\usepackage{amsmath}
\usepackage{inputenc}
\usepackage{amsfonts}
\usepackage{amssymb}
\usepackage{amsthm}
\usepackage{newlfont}
\usepackage{graphicx}
\usepackage{subfigure}
\usepackage{fullpage}
\newtheorem{theorem}{Theorem}

\newtheorem{corollary}{Corollary}

\inputencoding{ansinew}

\begin{document}

\setlength{\parindent}{0pt}

\title{About extension of upper semicontinuous multi-valued maps and applications}
\author{Y. Askoura}
\maketitle
Gretia, INRETS, 2 Avenue du Général Malleret-Joinville, 94114 Arcueil, France\\
email : askoura@inrets.fr,

\begin{abstract}
We formulate a multi-valued version of the Tietze-Urysohn extension
theorem. Precisely, we prove that any upper semicontinuous
multi-valued map with nonempty closed convex values defined on a
closed subset (resp. closed perfectly normal subset) of a completely
normal (resp. of a normal) space $X$ into the unit interval $[0,1]$
can be extended to the whole space $X$. The extension is upper
semicontinuous with nonempty closed convex values. We apply this
result for the extension of real semicontinuous functions, the
characterization of completely normal spaces, the existence of
Gale-Mas-Colell and Shafer-Sonnenschein type fixed point theorems
and the existence of equilibrium for qualitative games.
\end{abstract}



\section{Introduction}
The aim of this paper is to extend upper semicontinuous
(\textit{usc} in short) multi-valued maps over spaces as large as
possible. In particular, we want to avoid metrizability in
definition domains. We are interested by extending usc multi-valued
maps defined on a closed subset of a given topological space to
another. Similar results are already obtained by Cellina \cite{CEL},
Brodskii \cite{BRO1,BRO2}, Tan and Wu \cite{TW} and Ma \cite {MA},
using the metrizability of the domain. A more general result, in
this direction, is that of Borges \cite{BORGES}, which established
extensions for usc maps defined on closed subsets of stratifiable
spaces to any topological space. Cite also results of Drozdovskii
and Filippov \cite{DF} and Shishkov \cite{SHI} which extended usc
maps defined on closed subsets of paracompact completely normal
spaces to completely metrizable ones. Another type of extensions,
which is not concerned here, is to extend maps defined on dense
subsets \cite{HOL,LL}.

In the sequel, when speaking about extension of maps, we signify
extensions of the same type (single valued if the map is single
valued and multi-valued if the map is multi-valued) and preserving
the given continuity concept (continuity if the original map is
single valued and continuous, upper semicontinuity if the original
map is multi-valued and upper semicontinuous). If we try to compare
the extension of usc multi-valued maps with continuous single-valued
maps, two things appear. First, for usc maps, the extension need
only be usc, then we are tempted to say that the first extension is
easier. But, within the definition domains, for a map, to be
continuous and single-valued is very constraining comparatively with
the fact to be usc and multi-valued, then provides us additional
properties. So the two problems are, a priori, quite different, and
without evident comparison between them. The results of this paper
(and some of the cited ones) prove, in fact, that the extension of
multi-valued maps is more constraining. We consider an usc
multi-valued map $T:A\subset E\rightarrow I,$ where $A$ is a closed
subset of a topological space $E$ and $I$ the unit real interval. We
obtain an extension of $T$ when $E$ is completely normal or $E$ is
normal and $A$ is perfectly normal. As it is known, this type of
results provides a characterization of completely normal spaces.

An extension result (for multi-valued maps) in an infinite
uncountable product of spaces gives directly some existence
results of maximal elements and fixed points (\cite{GMC1},
\cite{SS} and \cite{GOU}) needed in game theory. The problem is to
avoid in the proof, properties which are not satisfied in an
uncountable products of usual (or a simply important class of)
spaces (like metrizability). Unfortunately, the complete normality
is not, in general, a property of an infinite uncountable product
of spaces. So, the results characterizing the complete normality
by extensions of usc multi-valued maps, suggest us to search
maximal elements (resp. equilibrium for qualitative games) for an
uncountable usc multi-valued maps (resp. with an uncountable set
of players and usc preference correspondences) with additional
requirements. Such applications are given in the second part of
this paper.

Our result can be seen as a multi-valued version of the
Tietze-Urysohn extension theorem. As it is known, the original proof
of this fundamental result uses the uniform convergence of a
sequence of functions. However, some researchers
(\cite{BON,MAI,OSS,SCO,TON}) asked for a possibility to prove this
result without the use of uniform convergence. The proof of our
extension theorem is inspired particularly by the paper of Ossa
\cite{OSS}, and proves that his technics are successful for usc
multi-valued maps. It is direct and elementary.

In this paper, \textit{usc} means upper semicontinuous, \textit{lsc}
means lower semicontinuous, $co(A)$ means the convex hull of $A$. If
$Y$ is a topological space and $X$ a subset of $Y$, then
$int_{Y}(X)$ refer to the interior of $X$ in $Y$ and $\overline{X}$
is the adherence (or the closure) of $X$ in $Y$. For a multi-valued
map $T:E_1\rightarrow E_2$, we denote $Dom(T)=\{x\in E_1,T(x)\neq
\emptyset\}$. In the whole of this document, the subsets are endowed
with the induced topology.
\section{Multi-valued version of the Tietze-Urysohn extension theorem}
Recall two definitions~:

A separated topological space $X$ is said to be completely normal
if it is hereditarily normal, that is : every subspace of $X$ is
normal. This definition is equivalent to the following : $X$ is
completely normal if and only if all subsets $A$ and $B$ of $X$
satisfying $\overline{A}\cap B=A\cap \overline{B}=\emptyset $ can
be separated by open sets, $i.e.$ there exists two open subsets of
$X,$ $U$ and $V,$ $A\subset U,$ $B\subset V$ such that $U\cap V=
\emptyset .$

A separated topological space $X$ is said to be perfectly normal
if each closed subset of $X$ is a $G_{\delta }$-set, $i.e.$
intersection of a countable open sets. Or equivalently, $X$ is
perfectly normal if each open subset of $X$ is an $F_{\sigma
}$-set, $i.e.\;$a countable union of closed sets.

The perfect normality imply complete normality and the converse is
false. See \cite{GAA} for more details about these notions.

The following theorem is the main result in this work.
\begin{theorem}
\label{T1} Let $X$ be a separated topological space, $A$ a closed subset of $%
X$ and $T:A\rightarrow \lbrack 0,1],$ a usc multi-valued map with closed
convex values. Suppose that one of the two conditions holds :

\begin{enumerate}
\item[$C1)$]  $X$ is completely normal,

\item[$C2)$]  $X$ is normal and $A$ is perfectly normal,
\end{enumerate}

Then, there exists a usc extension of $T$ with closed convex values defined
on $X$ into $[0,1].$ $i.e.$ $\exists \widetilde{T}:X\rightarrow \lbrack 0,1]$
usc with closed convex values such that $\widetilde{T}\mid _{A}\equiv T.$
\end{theorem}
\begin{proof}
Let $ B_{0}=\{0,1\},..., B_{n}=\{i/2^{n},i\in\{0,...,2^{n}\}\}$.
Define the set $B=\cup _{n\in \mathrm{I\!N}}B_{n}$ the set of all
dyadic numbers of $[0,1]$. We have, $B_{n+1}=B_{n}\cup
\{(r_{i}+r_{i+1})/2,i\in\{0,...,2^{n}-1\}\} ,r_{i}=i/2^{n}\in
B_{n}\}$. It is well known that $B$ is dense in $[0,1]$.
\newline
Let for every $r\in B,$ $A_{r}=T^{-1}([0,r])=\{x\in A,\exists y\in
T(x),y\leq r\}$. Since $T$ is usc on $A$, for all $r\in [0,1]$,
$A_r$ is closed.
\newline After this, we construct closed subsets $X_{r},r\in B,$
of $X$ satisfying the following three conditions :
\newline 1) $X_{r}\cap A=A_{r},$\newline 2) $int_{A}(A_{r})\subset
int_{X}(X_{r}),$\newline 3) $X_{r}\subset X_{s,}$ if $s,r\in B$
and $r<s.$\newline Put $X_{1}=X$. In the following step, the
condition $C1)$ or $C2)$ of the theorem is needed. We illustrate
the use of each of them. \newline Begin by the condition $C1)$
: the space $X$ is completely normal. We have, $\overline{int_{A}(A_{0})}%
\cap \complement_{A}A_{0}=int_{A}(A_{0})\cap
\overline{\complement_{A}A_{0}}=\emptyset $. Then,
we can separate $int_{A}(A_{0})$ and $\complement_{A}A_{0}$ by open sets (in $X$) $%
O_{0}$ and $O_{c}$ respectively. This gives $\overline{O_{0}}\cap
A\subset \complement_A (O_c\cap A) \subset A_{0}$. We conclude the
two relations $O_{0}\cap A=int_{A}(A_{0})$ and
$\overline{O_{0}}\cap A\subset A_{0}$, put $X_{0}=A_{0}\cup
\overline{O} _{0}$. Then, the sets $X_{r},$ $r\in B_{0},$
satisfying $1)-3)$ are defined.
\newline
For the same result, let us use the condition $C2)$ : $X$ is
normal and $A$ is
perfectly normal. In this case, $int_{A}(A_{0})$ and $\complement_{A}A_{0}$ are $%
F_{\sigma }$-sets, then, it can be written us a countable unions
of closed
sets, let $int_{A}(A_{0})=\underset{i\in \mathrm{I\!N}}{\cup }F_{i}$ and $%
\complement_{A}A_{0}=\underset{i\in \mathrm{I\!N}}{\cup }G_{i}.$ Furthermore, we have $%
\overline{int_{A}(A_{0})}\cap \complement_{A}A_{0}=int_{A}(A_{0})\cap \overline{%
\complement_{A}A_{0}}=\emptyset$. We apply a Bonan's lemma
\cite{BON} to separate $int_{A}(A_{0})$ and $\complement_{A}A_{0}$
by open sets, and we define $X_{0}$ by the same way.\newline Let
the sets $X_{r},r\in \cup _{k\leq n}B_{k}$ satisfying $1)-3)$ be
given. Then we obtain the sets $X_{r}$, $r\in B_{n+1}$ like this :
Let $r\in
B_{n+1}\backslash B_{n}$ and $i\in \{0,...,2^{n}-1\}$ such that $%
r=(r_{i}+r_{i+1})/2,$ with
$r_{i}=i/2^{n},r_{i+1}=(i+1)/2^{n}\;$are elements of $B_{n}.$
\newline We proceed as previously by the use of condition $C1)$ or $C2)$,
when constructing $X_{0}$, with the set $A_{r}$ in the place of
$A_{0}$. We obtain an open set $O_{r}^{\prime }$ of $X$ such that
$\overline{O_{r}^{\prime }}\cap A\subset
A_{r}$ and $O_{r}^{\prime }\cap A=int_{A}(A_{r}).$ Since $%
int_{A}(A_{r})\subset int_{A}(A_{r_{i+1}})\subset int_{X}(X_{r_{i+1}}),\;$%
the set $O_{r}=O_{r}^{\prime }\cap int_{X}(X_{r_{i+1}})$ is open
in $X$ and the two relations $\overline{O_{r}}\cap A\subset A_{r}$
and $O_{r}\cap A=int_{A}(A_{r})$ are obtained. Put
$X_{r}=A_{r}\cup \overline{O_{r}\cup X_{r_{i}}}.$ We easily verify
that conditions $1)-3)$ are satisfied (the verification is down
for the sets $X_{r_i},X_r$ and $X_{r_{i+1}}$). At
this moment, the recursive process for the construction of all the sets $%
X_{r},r\in B,$ satisfying conditions 1)-3) is given.\newline
Thereafter, we define the map $F:X\rightarrow \lbrack 0,1],$ as follows :%
\newline
$F(x)=\left\{
\begin{array}{l}
T(x)\text{ if }x\in A, \\
\inf \{r,x\in X_{r}\}\text{ otherwise.}
\end{array}
\right. $\newline Verify that $F$, as a map defined on $X$ is usc
at each point of $A$ ($i.e.$ with respect to the topology of $X$).
Let $x_{0}\in A$ and $a,b\in \lbrack 0,1]$ such that $
T(x_{0})=[t_{1},t_{2}]\subset ]a,b[$. The other cases are
analogous and explained next.
Note also that the convexity of the values of $T$ is considered here. Let $%
r_{1},r_{2}\in B$ such that $t_{2}<r_{1}<b$ and
$a<r_{2}<t_{1}.$\newline
In one hand, $x_{0}\in int_{A}(A_{r_{1}})$ (because $T$ is usc on $A$ and $%
T(x_{0})\subset \lbrack 0,r_{1}[$) and since
$int_{A}(A_{r_{1}})\subset int_{X}(X_{r_{1}})$, there exists an
open neighborhood $O_{1}$ of $x_{0}$ (in $X$) such that $\forall
x\in O_{1}\backslash A,$ $F(x)\leq r_{1}$. In other hand,
$x_{0}\notin A_{r_{2}}$. Since $x_0$ is an element of $A$,
$x_{0}\notin X_{r_{2}}$. Let $O_{2}$ be an open neighborhood of
$x_{0}$ in $X$ such that $ O\cap X_{r_{2}}=\emptyset $. We have,
$\forall x\in O_{2}\backslash A$, $F(x)\geq r_{2}.$ We take in the
last time an open set $O_{3}$ (in $X$) such that $\forall x\in
O_{3}\cap A,$ $T(x)\subset ]a,b[$ and define $O=O_{1}\cap
O_{2}\cap O_{3}$. We obtain $\forall x\in O,F(x)\subset ]a,b[,$
which gives the fact that $F$ is usc on $A$. The case of
$t_{2}=b=1$ and $T(x_0)\subset ]a,b]$ (resp. $t_{1}=a=0$ and
$T(x_0)\subset [a,b[$) is a simple particular case where it
suffices to consider only $O_{2}$ (resp. $O_{1} $). In the other
case, we put $O=X$.
\newline Denote by $H$ the graph of $F$ in
$X\times \lbrack 0,1]$. The desired map is $\widetilde{T}$ given
as follows~: $\widetilde{T}(x)=co\{y,(x,y)\in \overline{H}\}$.
This map is usc, because its graph is closed and the image space
($[0,1]$) is compact. We end by applying the upper semicontinuity
of $F$ on $A$ to prove that its values are not affected on $A$
when passing throw the closure of its graph.
\newline Let $x_{0}\in A.\;$We can separate $F(x_{0})=T(x_{0})$
and any point $y\notin T(x_{0})$ of $[0,1]$ by open sets $V_{0}$,
$V_{y}$ respectively. Then, there exists an open neighborhood
$O_{x_{0}}$ of $x_{0}$ such that $F(O_{x_{0}})\subset V_{0}.$ We
have obtained, $(x_{0},y)\in O_{x_{0}}\times V_{y},$
$F(O_{x_{0}})\subset V_{0}$ and $V_{y}\cap V_{0}=\emptyset ,$
which means that $O_{x_{0}}\times V_{y}\cap H=\emptyset .$ That is
$(x_{0},y)$ is not a point of the adherence of $H$. We can finally
affirm that : $\widetilde{T}\mid _{A}\equiv F\mid _{A}\equiv T$
\end{proof}

\medskip
Note that the previous theorem, stated only with condition $C1)$, is
proved differently by Shishkov \cite{SHI}.

\section{Applications} Now we give applications of Theorem
\ref{T1} in different domains. In the first time, analogously to
the characterization by Tietze-Urysohn extension theorem of normal
spaces, we give a characterization of completely normal spaces
(after Gutev \cite{GUT} and Shishkov \cite{SHI}, using Theorem
\ref{T1}). This proves that the last theorem, stated only with
condition $C1)$, can not be improved by the relaxation of the
complete normality imposed to $X$.
\begin{corollary}
\label{CC} Let $X$ be a separated topological space. Then, $X$ is
completely normal if and only if every usc multi-valued map $T$
defined on a closed subset of $X$ into $[0,1]$ with nonempty
closed convex values has a multi-valued usc extension defined on
the whole of $X$ into $[0,1]$ with nonempty closed convex values.
\end{corollary}
\begin{proof}
The necessity is a straightforward consequence of Theorem
\ref{T1}. The sufficiency is proved easily as follows : Let $A$
and $B$ be subsets of $X$ such that $A\cap \overline{B}=B\cap
\overline{A}=\emptyset .$ Define the multi-valued map
$T:\overline{A}\cup \overline{B}\rightarrow \lbrack 0,1],$ by $:$
$T(x)=\left\{
\begin{array}{l}
\lbrack 0,1]\text{ if }x\in \overline{A}\cap \overline{B}, \\
0\text{ if }x\in \overline{A}\backslash \overline{A}\cap \overline{B}, \\
1\text{ if }x\in \overline{B}\backslash \overline{A}\cap
\overline{B}.
\end{array}
\right. $ \newline
Let us verify that $T$ is usc on $\overline{A\cup B}$. Let $x_{0}\in $ $%
\overline{A\cup B}$ and $V$ an open subset of $[0,1]$ containing
$T(x_{0}).$ If $x_{0}\in \overline{A}\cap \overline{B},$ then
$V=[0,1]$, we can choose,
in this case, $O=\overline{A\cup B}$ as a neighborhood of $x_{0}$ such that $%
T(O)\subset V.$ If $x_{0}\in \overline{A}\backslash
\overline{A}\cap
\overline{B}= \complement_{\overline{A\cup B}}\overline{B}$, choose $%
O=\complement_{\overline{A\cup B}}\overline{B}$ as a neighborhood
of $x_{0}$ such that
$T(O)\subset V.$ The other case is similar. Then, $T$ is usc on $\overline{%
A\cup B}.$ \newline From the hypothesis, there exists a usc
extension $\widetilde{T}$ of $T$
defined on the whole of $X$, with closed convex values. The sets $\widetilde{T%
}_{-1}(]1/2,1])=\{x\in X,\widetilde{T}(x)\subset ]1/2,1]\}$ and $\widetilde{T%
}_{-1}([0,1/2[)$ are open and separate $A$ and $B.$ Then $X$ is
completely normal
\end{proof}

Theorem \ref{T1} can be applied to extend real semicontinuous
functions. In what follows, we use the same notations $usc$ and
$lsc$ for corresponding semicontinuity concepts for single-valued
real functions.
\begin{corollary}
Let $X$ be a separated topological space and $A$ a closed subset
of $X.$ Suppose that one of the conditions $C1)$ or $C2)$ is
satisfied. Given two functions $f,g:A\rightarrow \lbrack 0,1],$
such that $f$ is lsc, $g$ is usc and $f(x)\leq g(x)$, for every
$x\in A$. \newline Then, there exists extensions $\widetilde{f}$
and $\widetilde{g}$ of $f$ and $g$ respectively, such that
$\widetilde{f}$ is lsc, $\widetilde{g}$ is usc and
$\widetilde{f}(x)\leq \widetilde{g}(x)$, for every $x\in X$.
\end{corollary}
\begin{proof}
Define the multi-valued map $T:A\rightarrow \lbrack 0,1]$ by
$T(x)=[f(x),g(x)]$, for every $x\in A.$ Let us verify that $T$ is
usc on $A.$ Let $x_{0}\in A$ and $a,b\in \lbrack 0,1]$ such that
$T(x_{0})\subset ]a,b[.$ The other possibilities are particular
cases. We have $f(x_{0})>a$ and $g(x_{0})<b.$ Then, there exists
two neighborhoods $V_{1}$ and $V_{2}$ of $x_{0}$ (in $A$) such
that $f(x)>a,$ for every $x\in V_{1}$ and $g(x)<b$, for every
$x\in V_{2}.$ Put $V=V_{1}\cap V_{1}.$ This gives $T(V)\subset
]a,b[.$ If $a=0$, $b\neq 1$ and $T(x_0)\subset [a,b[$ (resp.
$b=1$, $a\neq 0$ and $T(x_0)\subset ]a,b]$), we do not need
$V_{1}$ (resp. $V_{2}$). If $a=0$, $b=1$ and $T(x_0)\subset
[a,b]$, we put $V=A.$\newline We infer, by Theorem \ref{T1}, an
usc extension $\widetilde{T}$ of $T$ with nonempty convex compact
values. Define $\widetilde{f}$ and $\widetilde{g}$ as follows~:
$\widetilde{f}(x)=\min \{y,y\in T(x)\}$ and $\widetilde{g}
(x)=\max \{y,y\in T(x)\}$, for every $x\in X.$
\newline It remains to verify the desired properties of
$\widetilde{f}$ and $ \widetilde{g}$. It is clear that
$\widetilde{f}_{\mid A}=f$, $\widetilde{g}_{\mid A}=g$ and
$\widetilde{f}(x)\leq \widetilde{g}(x),$ for every $x\in X$.
$\;$Let $\lambda \in [0,1]$. Then, $\{x\in X,\widetilde{f}(x)\leq
\lambda \}= \widetilde{T}^{-1}([0,\lambda ])$ and $\{x\in
X,\widetilde{g}(x)\geq \lambda \}=\widetilde{T}^{-1}([\lambda
,1])$.\ Since $\widetilde{T}$ is usc, the last level sets are
closed.\ This proves that $\widetilde{f}$ is lsc on $X$ and
$\widetilde{g}$ is usc on $X$
\end{proof}

We apply Theorem \ref{T1} to prove a version of the
Gale-Mas-Colell's \cite{GMC1}, Shafer-Sonnenschein's \cite{SS} and
Gourdel's \cite{GOU} fixed point theorems with an arbitrary number
(possibly uncountable) of multi-valued maps.

\begin{theorem}
\label{PFEM}Let $X_{\alpha },\alpha \in I$, $I$ is arbitrary and
$X_{\alpha}=[0,1]$. Let for every $\alpha$ in $I$, $T_{\alpha
}:\prod_{\lambda \in I}X_{\lambda }\rightarrow X_{\alpha }$ an usc
multi-valued map with empty or nonempty closed convex values such
that $Dom(T_{\alpha })$ is perfectly normal.

Then, $\exists x\in \prod_{\lambda \in I}X_{\lambda }$ such that : $\forall
\alpha \in I,$
\begin{equation*}
\text{either }x_{\alpha }\in T_{\alpha }(x)\text{ or }T_{\alpha
}(x)=\emptyset .
\end{equation*}
\end{theorem}
\begin{proof}
We apply Theorem \ref{T1} to find extensions
$\widetilde{T}_{\alpha }$ of $T_{\alpha },$ $\alpha \in I.$
Consider the map $F:\prod_{\lambda \in I}X_{\lambda }\rightarrow
\prod_{\lambda \in I}X_{\lambda },$ defined by$
F(x)=\prod_{\lambda \in I}\widetilde{T}_{\lambda }(x).$ Any fixed
point of $F $(apply any fixed point theorem for Kakutani maps in
locally convex spaces, for example Glicksberg's fixed point
theorem) ensure the result
\end{proof}

Now, we give an application to qualitative games with an
infinitely number of players. A qualitative game is a pair
$(X_{i},P_{i})_{i\in I},$ $I$ is the set of players, $X_{i}$ is
the set of strategies of the player $i\in I$ and $P_{i}$ is his
preference correspondence. For a literature in qualitative games,
see \cite{BOR,TW,CHE}.

\begin{theorem}
\label{QG} Let $G=(X_{i},P_{i})_{i\in I}$ be a qualitative game.
For every $i\in I$, $X_{i}=[0,1]$, $P_{i}:X=\prod_{j\in
I}X_{j}\rightarrow X_{i}$ is usc with empty or nonempty closed
convex values such that $Dom(P_{i})$ is perfectly normal and $I$
is an arbitrary set of indices. If $\forall i\in I,\forall x\in
X,x_{i}\notin P_{i}(x)$, then $G$ has an equilibrium, that is :
$\exists $ $y\in X,$ such that $\forall i\in I,P_{i}(y)=\emptyset
.$
\end{theorem}
\begin{proof}
Is a straightforward consequence of Theorem \ref{PFEM}
\end{proof}


\begin{thebibliography}{99}

\bibitem{BON}  E. Bonan, Sur un lemme adapt\'{e} au th\'{e}or\`{e}me de
Tietze-Urysohn. \textit{C.R. Acad. Sci. Paris, Ser. A.}, 270
(1970), 1226-1228.

\bibitem{BORGES} C. R. J. Borges, A study of multivalued
functions, \textit{Pacific J. Math.}, 23(1967), 451-461.

\bibitem{BOR}  A. Borglin and H. Keiding, Existence of equilibrium actions
and of equilibrium : a note on the ``new'' existence theorems,
\textit{J. Math. Eco.} 3 (1976), 313-316.

\bibitem{BRO1}  N. B. Brodskii, On extension of compact-valued maps,
\textit{Russian Math. Surveys}, 54(1999), 256-257.

\bibitem{BRO2} N. B. Brodskii, Extension of $UV^n$-valued mappings,
\textit{Math. Notes}, 66(1999), 283-291.

\bibitem{CEL}  A. Cellina, A theorem on the approximation of compact
multi-valued mappings, \textit{Rend. Sci. Fis. Mat. Nat., XLVII}(1969),
429-433.

\bibitem{CHE}  S. Chebbi and M. Florenzano, Maximal elements and equilibria
for condensing correspondences, \textit{Nonlinear Anal}., 38 (1999),
995-1002.

\bibitem{DF} S. A. Drozdovskii and V. V. Filippov, A selection
theorem for new class of set-valued maps, \textit{Math. Notes},
66(1999), 411-414.

\bibitem{GAA}  S. A. Gaal, \textit{Point set topology}, Academic Press,
1964, New York and London.

\bibitem{GMC1}  D. Gale and A. Mas-Colell, An equilibrium existence theorem
for a general model without ordered preferences, \textit{J. Math. Eco.},
2(1975), 9-15.


\bibitem{HOL} L. Hola, An extension theorem for multifunctions and a
characterization of complete metric spaces, \textit{Math.
Slovaca}, 38(1988), 177-182.

\bibitem{GOU}  P. Gourdel, Existence of intransitive equilibria in nonconvex
economies, \textit{Set-Valued Ana.}, 3(1995), 307-337.

\bibitem{GUT} V. Gutev, Generic extensions of finite-valued
u.s.c. selections, \textit{Topology Appl.}, 104(2000), 101-118.

\bibitem{LL} A. Lechicki and S. Levi, Extensions of semicontinuous
multifunctions, \textit{Forum Math.}, 2(1990), 341-360.

\bibitem{MA}  T. W. Ma, Topological degree of set-valued compact fields in
locally convex spaces, \textit{Diss. Math.}, 92(1972), 1-47.

\bibitem{MAI}  M. Mandelkern, A short proof of the Tietze-Urysohn extension
theorem, \textit{Arch., Math., }60(1993), 364-366.


\bibitem{OSS}  E. Ossa, A simple proof of the Tietze-Urysohn extension
theorem, \textit{Arch. Math.,} 71(1998), 331-332.

\bibitem{SCO}  B. M. Scott, A ``more topological'' proof of the
Tietze-Urysohn theorem, \textit{Amer. Math. Monthly,} 85 (1978),
192-193.

\bibitem{SS}  W. Shafer and H. Sonnenschein, Equilibrium in abstract
economies without ordered preferences, \textit{J. Math. Eco.}, 2(1975),
345-348.

\bibitem{SHI}  I. Shishkov, Hereditarily normal Kat\v etov spaces
and extending of usco mapping, \textit{Comment. Math. Univ.
Carolin.}, 41(2000), 183-195.

\bibitem{TW}  K.K. Tan and Z. Wu, An extension theorem and duals of
Gale-Mas-Colell's and Shafer-Sonnenschein's Theorems, \textit{Set
valued mappings with applications in nonlinear analysis, Ser.
Math. Anal. Appl.}, 4 (2002), 449-460, \textit{Taylor \& Francis,
London.}

\bibitem{TON}  H. Tong, Some characterization of normal and perfectely
normal spaces, \textit{Duke Math. J.}, 19 (1952), 289-292.


\end{thebibliography}
\end{document}